\documentclass{article}

\usepackage[english]{babel}

\usepackage[letterpaper,top=2cm,bottom=2cm,left=3cm,right=3cm,marginparwidth=1.75cm]{geometry}

\usepackage{algorithm}
\usepackage[noend]{algpseudocode}
\usepackage{mathtools}
\usepackage{amsmath}
\usepackage{mathrsfs}
\usepackage{amssymb}
\usepackage{amsfonts}
\usepackage{graphicx}
\usepackage[colorlinks=true, allcolors=blue]{hyperref}

\usepackage{amsopn}
\usepackage{amsthm}

\DeclareMathOperator*{\argmin}{arg\,min}

\newcommand{\E}{{\mathbb{E}}}
\newcommand{\R}{{\mathbb{R}}}

\usepackage{tikz}

\newtheorem{theorem}{Theorem}[section]
\newtheorem{lemma}[theorem]{Lemma}

\newtheorem{proposition}[theorem]{Proposition}

\newtheorem{assumption}[theorem]{Assumption}%
\newtheorem{definition}[theorem]{Definition}
\newtheorem{remark}[theorem]{Remark}

\title{A Note on Complexity for Two Classes of Structured Non-Smooth Non-Convex Compositional Optimization}
\author{Yao Yao\quad  Qihang Lin\quad Tianbao Yang}
\date{}

\begin{document}
\maketitle

\begin{abstract}
This note studies numerical methods for solving compositional optimization problems, where the inner function is smooth, and the outer function is Lipschitz continuous, non-smooth, and non-convex but exhibits one of two special structures that enable the design of efficient first-order methods. In the first structure, the outer function allows for an easily solvable proximal mapping. We demonstrate that, in this case, a smoothing compositional gradient method can find a $(\delta,\epsilon)$-stationary point—specifically defined for compositional optimization—in $\mathcal{O}(1/(\delta \epsilon^2))$ iterations. In the second structure, the outer function is expressed as a difference-of-convex function, where each convex component is simple enough to allow an efficiently solvable proximal linear subproblem. In this case, we show that a prox-linear method can find a nearly $\epsilon$-critical point in $\mathcal{O}(1/\epsilon^2)$ iterations.
\end{abstract}

\section{Introduction}
In this work, we consider the following \textit{compositional optimization problem}
\begin{equation}\label{eqn:prob1}
    \min_{x\in\R^d} f(x) := h(g(x)),
\end{equation}
where $h:\mathbb{R}^m\rightarrow \mathbb{R}$ is a Lipschitz continuous, non-smooth and non-convex function and $g:\R^d\rightarrow \R^m$ is a smooth map. Overall, (\ref{eqn:prob1}) is a non-smooth non-convex compositional optimization problem. 

This note focuses on first-order methods for \eqref{eqn:prob1} under some structural assumptions and analyzes their convergence. For general non-smooth non-convex optimization problems, both asymptotic convergence analysis, e.g., \cite{benaim2005, kiw07, majewski2018, Davis2020}, and non-asymptotic convergence analysis, e.g., \cite{kor22, zhang20}, have been established in literature for different first-order methods under various settings. Typically, non-asymptotic convergence analysis implies the iteration complexity of a first-order method for finding an $\epsilon$-stationary point. However, for general  non-smooth and non-convex optimization, it has been shown that a traditional $\epsilon$-stationary point cannot be obtained in finite complexity \cite{kor22, zhang20, tian2024no, jordan2023deterministic}. Therefore, recent studies have focused finding a \textit{$(\delta,\epsilon)$-stationary point}~\cite{zhang20,jordan2023deterministic} (also called a $(\delta,\epsilon)$-Goldstein stationary point). 

The first result of this note is a \emph{smoothing compositional gradient} method that finds a \textit{$(\delta,\epsilon)$-stationary} point for (\ref{eqn:prob1}) with iteration complexity $O(1/\delta \epsilon^2)$. The definition of such a $(\delta,\epsilon)$-stationary point (Definition \ref{def1}) is tailored for the compositional structure in (\ref{eqn:prob1}), and is similar to but different form that in \cite{zhang20,jordan2023deterministic}. 

In addition, we consider a special case of (\ref{eqn:prob1}), which is a difference-of-convex (DC) optimization with the following compositional structure.
\begin{equation}\label{eqn:prob2}
    \min_x f(x):= h_2(g_2(x))-h_1(g_1(x)),
\end{equation}
where $h_i:\mathbb{R}^m\rightarrow \mathbb{R}$ is a Lipschitz continuous convex function and $g_i:\R^d\rightarrow \R^m$ is a smooth map for $i=1,2$. To solve this problem, we propose a \emph{prox-linear} method where the solution is updated by solving two proximal linear subproblems. We show that, when the subproblems can be solved exactly, the method finds a nearly $\epsilon$-critical point of \eqref{eqn:prob2}, originally defined by \cite{Sun2022}, within $O(1/\epsilon^2)$ iterations.

\section{Related Works}


\textbf{Non-smooth non-convex problems.} General non-smooth non-convex problems and their associated algorithms have been systematically discussed in several monographs, including algorithms' asymptotic convergence \cite{benaim2005, majewski2018, Davis2020, bolte2020, Yang2024} and non-asymptotic convergence \cite{zhang20, kor22, alcantara2024,jordan2023deterministic}. Instead of considering a general problem, this note studies the algorithms for non-smooth non-convex optimization that exploit some special structure of problems. 

\textbf{Non-smooth compositional optimization.} One important structure is the compositional structure like the one in (\ref{eqn:prob1}). \cite{zhang2020optimal,wang2024alexr} consider the case when $h\circ g$ is convex while \cite{duchi2018, davis2019, Drusvyatskiy2019} assume convex $h$ and  smooth $g$, which ensures weak convexity of $h\circ g$. \cite{wang2024} consider a problem where a non-Lipschitz $\ell_p (0<p< 1)$ regularizer is added to the compositional function $h\circ g$ but still assume convexity of $h$ and smoothness of $g$. In a recent work \cite{hu23}, the authors consider the case when $h$ and $g$ are both weakly convex and $h$ is non-decreasing in each component, which also ensures weak convexity of $h\circ g$. On the contrary, we assume $h$ is a non-convex function which is not necessarily weakly convex, so (\ref{eqn:prob1}) is not necessarily weakly convex. \cite{pong2019} consider a similar setting to ours. Assuming $g$ is linear, they show the asymptotic convergence of a smoothing gradient method. This note considers nonlinear $g$ and provides a non-asymptotic convergence analysis of a similar algorithm. \cite{Bolte2018} considers the same problem as (\ref{eqn:prob1}) and they convert the compositional problem to a constrained optimization problem. To solve the equivalent constrained problem, they use Lagrangian-based method and establish algorithm's asymptotic convergence.

\textbf{Structured difference-of-convex problems.} The applications of DC optimization have arisen from statistical learning~\cite{nouiehed2019pervasiveness,ahn2017difference,cao2022unifying,liu2023risk,cui2018composite}, resource allocation~\cite{alvarado2014new}, and stochastic program~\cite{liu2020two}. The algorithms in~\cite{pang2017computing,pang2018decomposition} can be applied to find a stationary point stronger than the one considered in this note, but the convergence analysis in those works is asymptotic so the overall complexity for finding a (nearly) $\epsilon$-stationary is not provided in those works. Non-asymptotic convergence analysis for DC programs has been developed in many papers, including but not limited to \cite{yao22,Sun2022,moudaficomplete,Xu2018,tao1997convex,abbaszadehpeivasti2024rate,hu2024}. The difference-of-convex algorithm (DCA), e.g., the ones in \cite{tao1997convex,abbaszadehpeivasti2024rate,le2018convergence}, finds a nearly $\epsilon$-stationary point in $O(1/\epsilon^2)$ iterations, and can achieve a linear convergence when the Kurdyka-{\L}ojasiewicz (KL) property holds with certain exponents. However, the DCA requires exactly solving a non-trivial convex subproblem, which may not be easy in general. The smoothing methods based on Moreau envelopes~\cite{yao22,Sun2022,moudaficomplete} only need to solve a subproblem inexactly using an inner loop, and finds a nearly $\epsilon$-stationary point with $O(1/\epsilon^4)$ iteration complexity. 
DC optimization with a compositional structure like \eqref{eqn:prob2} have also been studied recently~\cite{thi2024,hu2024}. The authors of \cite{thi2024} proposed a method that linearizes $g_1(x)$ and $g_2(x)$ and then applies the DCA. The asymptotic convergence analysis is shown and they also consider DC compositional optimization problems with non-convex constraints. \cite{hu2024} studied a DC optimization problem formulated as 
$$
\min_x \left[ \max_{y \in Y} \varphi(x, y) - \max_{z \in Z} \psi(x, z) \right]
$$
where $\Phi(x) = \max_{y \in Y} \varphi(x, y)$ and $\Psi(x) = \max_{z \in Z} \psi(x, z)$ are weakly convex and $\varphi(x, y)$ and $\psi(x, z)$ are strongly concave in $y$ and $z$, respectively. They develop a single-loop algorithm for such a problem and show an $O(1/\epsilon^4)$ complexity for finding a nearly $\epsilon$-stationary point. If $h_1$ and $h_2$ in \eqref{eqn:prob2} are smooth so that their Fr\'{e}chet conjugates are strongly concave, (\ref{eqn:prob2}) becomes a special case of the problem above. In this note, we show that, when $g_1$ and $g_2$ are smooth and $h_1$ and $h_2$ are simple enough so the two proximal linear subproblems can be solved exactly, an analysis similar to \cite{hu2024} leads to $O(1/\epsilon^2)$ iteration complexity.

\textbf{Prox-linear method.} The prox-linear method~\cite{Drusvyatskiy2018, Drusvyatskiy2019, Drusvyatskiy2021, Lewis2016} is a well-studied method for a problem like~\eqref{eqn:prob1}.  Given a current iteration $x_k$, the prox-linear method generates the next iterate $x_{k+1}$ as an optimal or nearly optimal solution of 
\begin{equation*}
\min_{x} h(g(x_k)+\nabla g(x_k)(x-x_k))+\frac{1}{2t}\|x-x_k\|^2.
\end{equation*}
When $h$ is a finite-valued Lipschitz convex function and $g$ is a smooth mapping, this method finds an $\epsilon$-stationary point in $O(1/\epsilon^2)$ iteration, provided that the subproblem above is solved exactly. \cite{Drusvyatskiy2021} extends this algorithm by replacing the objective function in the subproblem above with a Taylor-like model.

\section{Notations and Preliminaries}
In this section, we present some notations and preliminaries. Throughout this work, a Euclidean space is considered and denoted by $\R^d$, with an inner product $\langle\cdot,\cdot\rangle$ and the induced norm $\|\cdot\|$. Let $B$ denote the unit Euclidean ball and $\text{conv}(\cdot)$ denote the convex hull. Given a linear map $A:\R^d\rightarrow \R^l$, the operator norm of $A$ is defined as $\|A\|_{\text{op}}:=\max_{\|u\|\leq 1}\|Au\|$. 
Given a set $S$ in $\R^d$, the distance of a point $x$ onto $S$ is given by
\begin{equation*}
    \text{dist}(x;S):= \inf_{y\in S}\|y-x\|.
\end{equation*}

For a non-convex function $f(x):\R^d\rightarrow \R$, let $\hat{\partial}f(x)$ denote the Fréchet subgradient and $\partial f(x)$ denote the limiting subgradient, i.e.,
\begin{equation*}
    \begin{split}
        \hat{\partial}f(\bar{x})&=\left\{ v\in\R^d:\lim_{x\rightarrow \bar{x}}\inf \frac{f(x)-f(\bar{x})-v^\top (x-\bar{x})}{\|x-\bar{x}\|}\geq 0  \right\},\\
        \partial f(\bar{x})&=\left\{ v\in\R^d:\exists x_k\xrightarrow{f} \bar{x},v_k\in \hat{\partial}f(\bar{x}), v_k\rightarrow v\right\},
    \end{split}
\end{equation*}
where the notation $x\xrightarrow{f} \bar{x}$ means that $x\rightarrow \bar{x}$ and $f(x)\rightarrow f(\bar{x})$. It is known that $\hat{\partial}f(\bar{x})\subset \partial f(x)$. 

For any proper function $f$ and any real number $\mu>0$, the \textit{Moreau envelope} and the \textit{proximal mapping} of $f$ are defined by
\begin{equation*}
    \begin{split}
        f_{\mu}(x) &:= \min_y \left\{ f(y)+\frac{1}{2\mu}\|y-x\|^2 \right\},\\
        \text{prox}_{\mu f}(x) & := \argmin_{y}\left\{ f(y)+\frac{1}{2\mu}\|y-x\|^2 \right\},
    \end{split}
\end{equation*}
respectively. It is not hard to see that
\begin{equation*}
   \min_{y} f(y)\leq  f_{\mu}(x) \leq f(x) \text{ for all } x.
\end{equation*}

\section{Outer Function with Easy  Mapping}\label{sec:comp}
In this section, we study \eqref{eqn:prob1} under the following assumptions.
\begin{assumption}
\label{asm:comp}
The following statements hold for \eqref{eqn:prob1}.
\begin{itemize}
    \item[1.] $h: \R^m\rightarrow \R$ is a $L_h$-Lipschitz continuous function:
    \begin{equation*}
        |h(x)-h(y)|\leq L_h\|x-y\|\quad \text{for all } x, y \in \R^m;
    \end{equation*}
    There exists a constant $\underline{h}$ such that $h(\cdot) \geq \underline{h}$ and $\text{prox}_{\mu h}(\cdot)$ can be solved easily.  
    \item[2. ] $g: \R^d\rightarrow \R^m$ is a $L_g$-Lipschitz continuous and $C^1$-smooth mapping with a $\beta$-Lipschitz continuous Jacobian map:
    \begin{equation*}
        \|\nabla g(x)-\nabla g(y)\|_{\text{op}}\leq \beta \|x-y\|\quad \text{for all } x, y \in \R^d;
    \end{equation*}
    There exists a constant $C_g$ such that $\|g(x)\|\leq C_g$ for all $x\in\R^d$.
\end{itemize}
\end{assumption}
To solve problem (\ref{eqn:prob1}), we first introduce the Moreau envelope of $h$ at $g(x)$:
\begin{equation*}
    h_{\mu}(g(x)):=\min_{y} \left\{h(y)+\frac{1}{2\mu}\|y-g(x)\|^2\right\},
\end{equation*}
and the proximal mapping of $h$ at $g(x)$:
\begin{equation*}
\text{prox}_{\mu h}(g(x)) := \text{arg}\min_{y} \left\{\frac{1}{2\mu}\|y-g(x)\|^2 + h(y)\right\},
\end{equation*}
where $\mu>0$. A nice property of the Moreau envelope function is that it can be written as a difference-of-convex (DC) function~\cite{pong2019,Xu2018}:
\begin{equation}\label{eqn:moreau}
    h_{\mu}(g(x))
     = \frac{1}{2\mu}\|g(x)\|^2-\underbrace{\max_{y}\left( \frac{1}{\mu}y^\top g(x)-\frac{1}{2\mu}\|y\|^2-h(y) \right)}_{G(x)},
\end{equation}
where $G(x)$ is a weakly convex function because it is the max of weakly convex functions of $x$. Also, if the assumptions above are satisfied, it is easy to verify that $\frac{1}{2\mu}\|g(x)\|^2$ is $\frac{L_g^2+C_g\beta}{\mu}$-smooth and hence $\frac{L_g^2+C_g\beta}{\mu}$-weakly convex. 
We also present following nice properties that are useful.
\begin{lemma}\label{lemma1}
Under Assumption~\ref{asm:comp}, it holds that
    \begin{eqnarray}
            &&\text{prox}_{\mu h}(g(x)) \subset \left( I + \mu \partial h \right)^{-1}(g(x))\ \text{for all } x,\label{eqn:lemma1_1}\\
            &&\forall v \in \text{prox}_{\mu h}(g(x)),\ \frac{1}{\mu}(g(x)-v)\in \partial h(v).\label{eqn:lemma1_2}
    \end{eqnarray}
\end{lemma}
The proof of (\ref{eqn:lemma1_1}) can be found in \cite[Example 10.2]{Rock98} and the proof of (\ref{eqn:lemma1_2}) can be found in \cite[Theorem 10.1]{Rock98}.
\begin{lemma}\label{lemma2}
    If Assumption~\ref{asm:comp} holds, for any $v \in \text{prox}_{\mu h}(g(x))$,
    \begin{equation*}
        \|v\| \leq C_v,
    \end{equation*}
    where $C_v:= C_g +\sqrt{2\mu\left(\max_x h(g(x))-\underline{h}\right)}$.
\end{lemma}
\begin{proof}
    Since $h_{\mu}(g(x))\leq h(g(x))$ and $h$ is lower bounded by $\underline{h}$, we can get for any $v \in \text{prox}_{\mu h}(g(x))$
    \begin{equation*}
        \underline{h}+\frac{1}{2\mu}\|v-g(x)\|^2\leq h(v)+\frac{1}{2\mu}\|v-g(x)\|^2\leq h(g(x)).
    \end{equation*}
    By rearranging the above inequalities, we have
    \begin{equation*}
        \|v-g(x)\|^2 \leq 2\mu (h(g(x))-\underline{h})\leq 2\mu\left(\max_x h(g(x))-\underline{h}\right).
    \end{equation*}
    By the $L_h$-Lipschitz continuity of $h$ and the upper boundness of $\|g(x)\|$, we know that the $\max_x h(g(x))$ exists.
    Then by triangle inequality, we can get
    \begin{equation*}
    \begin{split}
        \|v\| &\leq \|g(x)\|+\sqrt{2\mu\left(\max_x h(g(x))-\underline{h}\right)}\\
        & \leq C_g +\sqrt{2\mu\left(\max_x h(g(x))-\underline{h}\right)}.
    \end{split}
    \end{equation*}
    The result follows.
\end{proof}

\subsection{Targeted solution}
Consider the problem (\ref{eqn:prob1}), the Fréchet subgradient of $f$ can be get following a chain-rule \cite[Theorem 10.6]{Rock98}.
\begin{lemma}[Chain rule]
\label{thm:chainrule}
For $f$ defined in (\ref{eqn:prob1}), at any $x$, one has
\begin{equation*}
   \nabla g(x)^\top\hat{\partial}h(g(x))\subset   \hat{\partial} f(x) \subset   \partial f(x) \subset \nabla g(x)^\top \partial h(g(x)).
\end{equation*}
\end{lemma}
Ideally, we want to find a point $x$ such that $\min\left\{\|\xi\|\ |\  \xi\in \partial f(x)\right\}\leq \epsilon$. However, the existing results show that this cannot be done with finite complexity~\cite{kor22, zhang20, tian2024no, jordan2023deterministic}. Instead, one can find a $(\delta,\epsilon)$-Goldstein stationary point using first-order methods~\cite{zhang20,jordan2023deterministic} with iteration complexity $O(1/(\delta\epsilon^3))$.   Here, a $(\delta,\epsilon)$-Goldstein stationary point is a point $x$ satisfying 
\begin{equation}
\label{eq:Goldstein}
 \text{dist}(0, \partial f(x+\delta B))\leq \epsilon
\end{equation}
where $\partial f(x+\delta B):= \text{conv}\left(\bigcup_{y\in x+\delta B}\partial f(x)\right)$ is the Goldstein $\delta$-subdifferential \cite{Goldstein1977}.

According to the last inclusion in Lemma~\ref{thm:chainrule}, if we weaken a $(\delta,\epsilon)$-Goldstein stationary point in \eqref{eq:Goldstein} to a $(\delta,\epsilon)$-stationary point defined with respect to set $\nabla g(x)^\top \partial h(g(x))$, it is possible to find such a weakened stationary point more easily, i.e., with lower iteration complexity. To verify this, the point we want to find through this study is defined below.
\begin{definition}\label{def1}
     A point $x$ is called a (chain-rule) \textit{$(\delta,\epsilon)$-stationary} point for  (\ref{eqn:prob1}) if 
    \begin{equation}
        \text{dist}(0, \nabla g(x)^\top\partial h(g(x)+\delta B))\leq \epsilon,
    \end{equation}
    where $\partial h(g(x)+\delta B):= \text{conv}\left(\bigcup_{y\in g(x)+\delta B}\partial h(y)\right)$ is the Goldstein $\delta$-subdifferential, introduced in \cite{Goldstein1977}.
   A random variable $x$ is called a \textit{stochastic (chain-rule) $(\delta,\epsilon)$-stationary} point for (\ref{eqn:prob1}) if 
    \begin{equation}
        \E\left[\text{dist}(0, \nabla g(x)^\top\partial h(g(x)+\delta B))\right]\leq \epsilon.
    \end{equation}
\end{definition}

\begin{remark}
To interpret Definition~\ref{def1}, we write (\ref{eqn:prob1}) equivalently into a constrained problem 
\begin{equation}\label{eqn:prob1_constrained}
    \min_{x\in\R^d,v\in\mathbb{R}^m} h(v)\text{ s.t. }g(x)=v.
\end{equation}
According to the optimality conditions of this problem, we can say $(x,v)$ is a Karush-Kuhn-Tucker (KKT) point of \eqref{eqn:prob1_constrained} if there exists $\lambda\in\mathbb{R}^m$ such that $\lambda\in \partial h(v)$, $\nabla g(x)^\top \lambda=0$,  and $v=g(x)$. Suppose $x$ is a $(\delta,\epsilon)$-stationary point in Definition~\ref{def1} for (\ref{eqn:prob1}) with $\delta=O(\epsilon)$. There exist $v\in\mathbb{R}^m$ and $\lambda\in \partial h(v)$ such that $\|g(x)-v\|\leq \delta=O(\epsilon)$ and $\|\nabla g(x)^\top\lambda\|\leq \epsilon$, indicating that $(x,v)$ is an $\epsilon$-KKT point of \eqref{eqn:prob1_constrained}.
\end{remark}

\subsection{Algorithm and  convergence analysis}

To find a $(\delta,\epsilon)$-stationary point of \eqref{eqn:prob1}, we propose a smoothing compositional gradient method (SCGM) in Algorithm \ref{alg:spgm}. This algorithm is essentially a gradient method applied to minimize $h_{\mu}(g(x))$ in \eqref{eqn:moreau}, which is a smooth approximation of \eqref{eqn:prob1}. Since we have asssumed the proximal mapping $\text{prox}_{\mu h}(g(x))$ can be easily solved, we can get $v(g(x_k))\in \text{prox}_{\mu h}(g(x_k))$ for $k=1, \cdots, K$. Note that $v(g(x_k))$ may not be unique since $h(\cdot)$ is not necessarily weakly convex. If there are multiple possible $v(g(x_k))$, our algorithm can use any one of them. Recall (\ref{eqn:moreau}). At the $k$th iteration, after getting one proximal point $v(g(x_k))$, we can get
\begin{equation*}
    \frac{1}{\mu}\nabla g(x_k)^\top g(x_k)-\frac{1}{\mu}\nabla g(x_k)^\top v(g(x_k))\in \partial h_{\mu}(g(x_k)),
\end{equation*}
and then we apply subgradient update on $x$ to generate $x_{k+1}$.

\begin{algorithm}
\caption{Smoothing Compositional Gradient Method (SCGM)}\label{alg:spgm}
\begin{algorithmic}[1]
\State{\textbf{Initialization:} $x_1$, $K\in \mathbb{Z}_+$, $\mu>0, \gamma > 0$.}
\For{$k=1,\cdots,K$}
\State{Get $v(g(x_k))\in \text{prox}_{\mu h}(g(x_k))$}
\State{Let $F_k(x) = \left(\frac{1}{\mu}\nabla g(x_k)^\top g(x_k) - \frac{1}{\mu}\nabla g(x_k)^\top v(g(x_k))\right)^\top(x-x_k)+\frac{\gamma}{2}\|x-x_k\|^2$}
\State{$x_{k+1}= \argmin\limits_{x} F_k(x)$}
\EndFor
\textbf{end for}
\State{\textbf{return} $x_\tau$ with $\tau$ uniformly sampled from $\{1, \cdots, K\}$}
\end{algorithmic}
\end{algorithm}

The convergence property of Algorithm \ref{alg:spgm} is presented as follows.
\begin{theorem}
    Suppose Assumption~\ref{asm:comp} holds and $\mu=\frac{\delta}{2L_h}$, $\gamma=\frac{2C}{\mu}$, and $K= \frac{16L_hC\Delta}{\delta\epsilon^2}$ with $C:=L_g^2+\beta C_g+C_vL_g$ and $\Delta:=h(g(x_1))-\underline{h}$, then the output of the Algorithm \ref{alg:spgm} $x_{\tau}$ is a stochastic $(\delta,\epsilon)$-stationary point for (\ref{eqn:prob1}).
\end{theorem}

\begin{proof}
Since $\frac{1}{2\mu}\|g(x)\|^2$ is $\frac{L_g^2+\beta C_g}{\mu}$-smooth, we have for any $u\in \mathbb{R}^d$,
\begin{equation}\label{eqn:g^2}
    \frac{1}{2\mu}\|g(u)\|^2\leq \frac{1}{2\mu}\|g(x_k)\|^2+\left\langle\frac{1}{\mu}\nabla g(x_k)^\top g(x_k), u-x_k\right\rangle+\frac{L_g^2+\beta C_g}{2\mu}\|u-x_k\|^2.
\end{equation}
By the Lemma \ref{lemma2}, we observe that $G(x)$ in (\ref{eqn:moreau}) is $\frac{C_vL_g}{\mu}$-weakly convex with $C_v$ defined in Lemma \ref{lemma2}. Then we have for any $u\in\mathbb{R}^d$,
\begin{equation}\label{eqn:G}
    -G(u) \leq -G(x_k)-\left\langle\frac{1}{\mu}\nabla g(x_k)^\top v(g(x_k)), u-x_k\right\rangle+\frac{C_vL_g}{2\mu}\|u-x_k\|^2.
\end{equation}
Combine (\ref{eqn:g^2}) and (\ref{eqn:G}) to get
\begin{equation*}
    h_{\mu}(g(u))\leq h_{\mu}(g(x_k))+\left\langle \frac{1}{\mu}\nabla g(x_k)^\top (g(x_k)-v(g(x_k))),u-x_k\right\rangle+\frac{L_g^2+\beta C_g+C_vL_g}{2\mu}\|u-x_k\|^2.
\end{equation*}
Let $C:=L_g^2+\beta C_g+C_vL_g$ and replace $u$ with $x_{k+1}$ in above inequality, we have
\begin{equation}\label{eqn:h1}
    h_{\mu}(g(x_{k+1}))\leq h_{\mu}(g(x_k))+\left\langle \frac{1}{\mu}\nabla g(x_k)^\top (g(x_k)-v(g(x_k))),x_{k+1}-x_k\right\rangle+\frac{C}{2\mu}\|x_{k+1}-x_k\|^2.
\end{equation}
Since $x_{k+1}=\argmin_{x} F_k(x)$, we have $F_k(x_{k+1})\leq F_k(x_k)$, i.e.,
\begin{equation}\label{eqn:h2}
    \left\langle \frac{1}{\mu}\nabla g(x_k)^\top (g(x_k)-v(g(x_k))),x_{k+1}-x_k\right\rangle+\frac{\gamma}{2}\|x_{k+1}-x_k\|^2 \leq 0
\end{equation}
and
\begin{equation}\label{eqn:x_{k+1}_2}
   \frac{1}{\mu}\nabla g(x_k)^\top g(x_k) - \frac{1}{\mu}\nabla g(x_k)^\top v(g(x_k))+ \gamma(x_{k+1}-x_k) = 0.
\end{equation}
By adding (\ref{eqn:h2}) to (\ref{eqn:h1}), we have
\begin{equation}\label{eqn:x_{k+1}_1}
    h_{\mu}(g(x_{k+1}))\leq h_{\mu}(g(x_k)) + \left(\frac{C}{2\mu}-\frac{\gamma}{2}\right)\|x_{k+1}-x_k\|^2.
\end{equation}
By letting $\gamma = \frac{2C}{\mu}$, (\ref{eqn:x_{k+1}_1}) can be easily converted to
\begin{equation*}
    \frac{C}{2\mu}\|x_{k+1}-x_{k}\|^2\leq h_{\mu}(g(x_k)) - h_{\mu}(g(x_{k+1})).
\end{equation*}
Rearrange (\ref{eqn:x_{k+1}_2}) and combine with the above formula, we can get 
\begin{equation*}
\begin{split}
    \frac{1}{\mu^2}\|\nabla g(x_k)^\top g(x_k) - \nabla g(x_k)^\top v(g(x_k))\|^2=&\gamma^2\|x_{k+1}-x_k\|^2\\
    \leq & 4\gamma (h_{\mu}(g(x_k)) - h_{\mu}(g(x_{k+1})))
\end{split}    
\end{equation*}
By taking average over $1,\cdots,K$ for above inequality, we have
\begin{equation*}
\begin{split}
    \frac{1}{K}\sum^K_{k=1}\frac{1}{\mu^2}\|\nabla g(x_k)^\top g(x_k) - \nabla g(x_k)^\top v(g(x_k))\|^2 &\leq 4\gamma \frac{h_{\mu}(g(x_1)) - h_{\mu}(g(x_{K}))}{K}\\
    &\leq 4\gamma \frac{h_{\mu}(g(x_1)) - \inf_x h_{\mu}(g(x))}{K}\\
    & \leq \frac{4\Delta \gamma}{K},
\end{split}
\end{equation*}
where $\Delta := h(g(x_1))-\underline{h}$. Here, we have used the fact that $h_{\mu}(g(x_1)) - \inf_x h_{\mu}(g(x))\leq \Delta$. Indeed, by the definition of $h_{\mu}(g(x))$, we have $\inf_x h_{\mu}(g(x))\geq \min_x h(x)\geq \underline{h}$ and $h_{\mu}(g(x_1))\leq h(g(x_1))$. Hence, we have
\begin{equation*}
    h_{\mu}(g(x_1)) - \inf_x h_{\mu}(g(x)) \leq h_{\mu}(g(x_1)) - \underline{h}\leq h(g(x_1))- \underline{h}=\Delta.
\end{equation*}
Then by the definition of $x_{\tau}$, we have
\begin{equation*}
    \mathbb{E}\left[ \frac{1}{\mu^2}\left\|\nabla g(x_{\tau})^\top\left(g(x_{\tau})- v(g(x_{\tau}))\right)\right\|^2 \right] \leq \frac{4\Delta \gamma}{K}.
\end{equation*}
Next, we want to bound the distance between $v(g(x_{\tau}))$ and $g(x_{\tau})$. By the definition of $v(g(x_{\tau}))$, we have
\begin{equation*}
     h(v(g(x_{\tau}))+\frac{1}{2\mu}\|v(g(x_{\tau}))-g(x_{\tau})\|^2\leq h(g(x_{\tau})),
\end{equation*}
from which we can get
\begin{equation*}
    \frac{1}{2\mu}\|v(g(x_{\tau}))-g(x_{\tau})\|^2\leq h(g(x_{\tau}))-h(v(g(x_{\tau})))\leq L_h \|g(x_{\tau})-v(g(x_{\tau}))\|,
\end{equation*}
where the last inequality comes from the $L_h$-Lipschitz continuity of $h$. From the above inequality, we can get
\begin{equation*}
    \|g(x_{\tau})-v(g(x_{\tau}))\| \leq 2\mu L_h.
\end{equation*}
By setting $\mu = \frac{\delta}{2L_h}$ and $K = \frac{16L_hC\Delta}{\delta \epsilon^2}$, we can get
\begin{equation*}
\begin{split}
    \mathbb{E}\left[ \frac{1}{\mu}\left\|\nabla g(x_{\tau})^\top\left(g(x_{\tau})- v(g(x_{\tau}))\right)\right\| \right] &\leq \epsilon,\\
    \|g(x_{\tau})-v(g(x_{\tau}))\| &\leq \delta.
\end{split}
\end{equation*}
By combining the above results with the Lemma (\ref{lemma1}), we can get
\begin{equation*}
    \E\left[ \text{dist}(0, \nabla g(x_{\tau})^\top \partial h(v(g(x_{\tau}))) \right] \leq \epsilon
\end{equation*}
with $v(g(x_{\tau}))\in g(x_{\tau})+\delta B$, which shows that $x_{\tau}$ is a stochastic $(\delta,\epsilon)$-stationary point for the problem (\ref{eqn:prob1}).
\end{proof}

This result shows that the iteration complexity of finding a (chain-rule) $(\delta,\epsilon)$-stationary point is lower than that of $(\delta,\epsilon)$-Goldstein stationary point by a factor of $O(1/\epsilon)$.

\section{Difference-of-Convex Outer Function}
In this section, we consider \eqref{eqn:prob2}, which is a special case of \eqref{eqn:prob1} called compositional DC optimization. We define $f_1(x):= h_1(g_1(x))$ and $f_1(x):= h_2(g_2(x))$ so \eqref{eqn:prob2} can be rewritten as
\begin{equation*}
\begin{split}
    \min_x f(x) & := f_2(x)-f_1(x)= h_2(g_2(x))-h_1(g_1(x)).
\end{split}
\end{equation*}
Throughout this section, we make the following assumptions on \eqref{eqn:prob2}.
\begin{assumption}
\label{asm:dc}
The following statements hold for \eqref{eqn:prob2}.
\begin{itemize}
    \item[1.] $h_i:\R^m\rightarrow \mathbb{R}$ is a convex and $L$-Lipschitz continuous function 
    and $\text{prox}_{\mu h_i(Ax+b)}(\cdot)$ can be solved easily for $i=1,2$ and for any $A\in\mathbb{R}^{m\times d}$ and any $b\in\mathbb{R}^m$.  
    \item[2.] $g_i: \mathbb{R}^d\rightarrow \mathbb{R}^m$ is a $C^1$-smooth mapping with a $\beta$-Lipschitz continuous Jacobian map for $i=1,2$.  There exists a constant $C_g$ such that $\max\{\|g_1(x)\|, \|g_2(x)\|\}\leq C_g$ for all $x\in\R^d$.
\end{itemize}
\end{assumption}
\begin{lemma}[Lemma 4.2 in \cite{Drusvyatskiy2019}]
\label{lemma:weak_con}
Under Assumption~\ref{asm:dc}, $f_1$ and $f_2$ are $\rho$-weakly convex with 
$$
\rho:= L\beta.
$$
\end{lemma}
By the Lemma \ref{lemma:weak_con}, problem \eqref{eqn:prob2} is DC optimization where the two convex components can be $f_2(x)+\frac{1}{2\mu}\|x\|^2$ and $f_1(x)+\frac{1}{2\mu}\|x\|^2$ for any $\mu\in (0,\rho^{-1})$.

\subsection{Targeted solution}
We say $x\in\R^d$ is a \textit{critical point} of (\ref{eqn:prob2}) if $0\in \partial f_2(x)-\partial f_1(x)$. Given $\epsilon > 0$, we say $x\in\R^d$ is an \textit{$\epsilon$-critical point} of (\ref{eqn:prob2}) if there exist $\xi_1\in \partial f_1(x)$ and $\xi_2 \in \partial f_2(x)$ such that $\|\xi_1-\xi_2\|\leq \epsilon$. Since computing an $\epsilon$-critical point is difficult, following recent literature \cite{Sun2022,yao22,hu2024},  we instead consider finding a \emph{nearly $\epsilon$-critical point} defined below.
\begin{definition}\label{def2}
   Given $\epsilon>0$, we say $x\in\R^d$ is a nearly $\epsilon$-critical point of (\ref{eqn:prob2}) if there exist $\xi_1$, $\xi_2$, $x'$, and $x'' \in \R^d$ such that $\xi_1\in \partial f_1(x')$, $\xi_2 \in \partial f_2(x'')$ and $\max\{\|\xi_1-\xi_2\|, \|x-x'\|, \|x-x''\|\}\leq \epsilon$.    A random variable $x$ is called a \textit{stochastic nearly $\epsilon$-critical} point for (\ref{eqn:prob2}) if there exist random variables $\xi_1$, $\xi_2$, $x'$, and $x'' \in \R^d$ such that $\xi_1\in \partial f_1(x')$, $\xi_2 \in \partial f_2(x'')$ and $\max\{\E\|\xi_1-\xi_2\|, \E\|x-x'\|, \E\|x-x''\|\}\leq \epsilon$. 
\end{definition}
\begin{remark}
To interpret Definition~\ref{def2}, we write (\ref{eqn:prob2}) equivalently into a constrained problem 
\begin{equation}\label{eqn:prob2_constrained}
    \min_{x_1,x_2\in\R^d} f_2(x_2)-f_1(x_1)\text{ s.t. }x_1=x_2.
\end{equation}
This is still DC optimization as it can be represented as $[f_2(x_2)+\mathbf{1}_{x_1=x_2}(x_1,x_2)] - f_1(x_1)$, where $\mathbf{1}_{x_1=x_2}(x_1,x_2)$ is an indicator function that equals zero if $x_1=x_2$ and equals positive infinity if $x_1\neq x_2$. We then can say $(x_1,x_2)$ is a critical point of \eqref{eqn:prob2_constrained} if $x_1=x_2$ and there exists $\lambda \in\mathbb{R}^m$  such that 
$$
\partial [f_2(x_2)+\mathbf{1}_{x_1=x_2}(x_1,x_2)]\cap \partial f_1(x_1) = \{\lambda\}\times (\partial f_2(x_2)-\lambda)\cap f_1(x_1)\times\{0\}\neq\emptyset.$$
These conditions can be stated as $\lambda\in \partial f_2(x_2)$, $\lambda\in\partial f_1(x_1)$ and $x_2=x_1$, which can be viewed as KKT conditions for \eqref{eqn:prob2_constrained} with  $\lambda$ being  a Lagrangian multiplier. Then, suppose $x$ is a nearly  $\epsilon$-critical point with $\xi_1$, $\xi_2$, $x'$, and $x'' \in \R^d$ defined in Definition~\ref{def2}. We then have $\xi_2\in \partial f_2(x'')$, $\text{dist}(\xi_2,\partial f_1(x'))\leq \|\xi_2-\xi_1\|\leq \epsilon$, and $\|x'-x''\|\leq \|x'-x\|+ \|x-x''\|\leq 2\epsilon$, suggesting that $(x',x'')$ is a $2\epsilon$-KKT point of \eqref{eqn:prob2_constrained} and thus $(x,x)$ can be called a nearly  $2\epsilon$-KKT point of \eqref{eqn:prob2_constrained} in the sense that $\|(x,x)-(x',x'')\|\leq 2\epsilon$.
\end{remark}

Following \cite{Sun2022,yao22,hu2024}, to solve (\ref{eqn:prob2}), we approximate $f(x)$ by a smooth function using the Moreau envelopes of $f_1$ and $f_2$, which are defined as 
\begin{equation}
\label{eq:envelopes12}
    f_{1,\mu} (z) := \min_x\left\{ h_{1}(g_1 (x)) + \frac{1}{2\mu}\|x-z\|^2 \right\}\quad\text{and}\quad
    f_{2,\mu} (z) := \min_x\left\{ h_{2}(g_2 (x)) + \frac{1}{2\mu}\|x-z\|^2 \right\}
\end{equation}
where $\mu\in (0,\rho^{-1})$. The proximal mappings of $f_1$ and $f_2$ are defined as
\begin{equation}
\label{eq:proximal12}
    x_1^*(z) := \argmin_x\left\{ h_{1}(g_1 (x)) + \frac{1}{2\mu}\|x-z\|^2 \right\}\quad\text{and}\quad
    x_2^*(z) := \argmin_x\left\{ h_{2}(g_2 (x)) + \frac{1}{2\mu}\|x-z\|^2 \right\}.
\end{equation}
Note that  $x_1^*(z)$ and $x_2^*(z)$ are unique because the minimization problems above are strongly convex. It is well known that $f_{i,\mu}(x)$ is smooth with $\nabla f_{i,\mu}(z) = \mu^{-1}(z-x_i^*(z))$ \cite[Proposition 13.37]{Rock98}. Moreover, $x_i^*(z)$ is $(1-\mu\rho)^{-1}$-Lipschitz continuous by the following lemma.
\begin{lemma}[Lemma 3.5 in \cite{Zhang2020}]\label{lemma_prox}
    For any $z, z'\in \mathbb{R}^d$, we have
    \begin{equation*}\label{eqn:lemma}
        \|x_i^*(z)-x_i^*(z')\|\leq (1-\mu\rho)^{-1} \|z-z'\|,\quad i=1,2.
    \end{equation*}
\end{lemma}
Hence, using the Moreau envelopes, we can construct a smooth approximation of (\ref{eqn:prob2}) as follows
\begin{equation}\label{eqn:smooth_prob2}
    \min_z \left\{ f_{\mu}(z):= f_{2,\mu}(z)- f_{1,\mu}(z)\right\}.
\end{equation}
By the aforementioned properties of $f_{1,\mu}$ and $f_{2,\mu}$, function $f_{\mu}$ is smooth with $\nabla f_{\mu}(z)=\mu^{-1}(x_1^*(z)-x_2^*(z))$ and $\nabla f_{\mu}$ is $L_{\mu}$-Lipschitz continuous with $L_{\mu}=\frac{2}{\mu-\mu^2\rho}$.  

\subsection{Algorithm and convergence analysis}
After approximating (\ref{eqn:prob2}) by \eqref{eqn:smooth_prob2}, we can directly apply a first-order method for smooth non-convex optimization to (\ref{eqn:smooth_prob2}). Suppose the solution for (\ref{eqn:smooth_prob2}) at iteration $k$ is $z^k$. To compute $\nabla f_{\mu}(z^k)$, we need to compute $x_1^*(z^k)$ and $x_2^*(z^k)$ exactly, which is not easy. Suppose, during the last iteration, we have obtained some approximations to $x_i^*(z^{k-1})$, denoted by $x_1^k$, for $i=1$ and $2$. Then we adopt the key idea of the single-loop prox-linear approximate gradient method (PAGM) where $g_i(x)$ and term $\frac{1}{2\mu}\|x-z^k\|^2$ in \eqref{eq:envelopes12} are linearized at $x=x_i^k$ for $i=1$ and $2$, and the following prox-linear subproblems are solved.
\begin{equation}
x_i^{k+1}=\argmin\limits_{x} \left\{h_i(g_i(x^k_i)+\nabla g_i(x^k_i)^\top (x-x_i^k))+\frac{1}{\mu}\langle x_i^k-z^k, x-x_i^k\rangle+\frac{1}{2t}\|x-x_i^k\|^2\right\},\quad i=1,2,
\end{equation}
where $t>0$ is a step size. Note that $x_i^{k+1}$ can be obtained easily thanks to Assumption~\ref{asm:dc}. Then  $x_1^{k+1}$ and $x_2^{k+1}$ are used as approximations to $x_1^*(z^k)$ and $x_2^*(z^k)$, respectively, so we obtain an approximation of $\nabla f_{\mu}(z^k)=\mu^{-1}(x_1^*(z^k)-x_2^*(z^k))$ as $\mu^{-1}(x_1^{k+1}-x_2^{k+1})$. Then $z^k$ can be updated to $z^{k+1}$ through a gradient descend step along the direction of this approximate gradient, that is, $z^{k+1} = z^k - \gamma\mu^{-1}(x_1^{k+1}-x_2^{k+1})$ where $\gamma>0$ is a step size. This procedure is formally stated in Algorithm \ref{alg:plm}.

\begin{algorithm}
\caption{Prox-Linear Approximate Gradient Method (PAGM)}\label{alg:plm}
\begin{algorithmic}[1]
\State{\textbf{Initialization:} $x_1^1$, $x_2^1$, $z^1$, $K\in \mathbb{Z}_+$, $t>0$, $\mu>0$, $\gamma>0$.}
\For{$k=1,\cdots,K$}
\State{$x_1^{k+1}=\argmin\limits_{x} \left\{h_1(g_1(x^k_1)+\nabla g_1(x^k_1)^\top (x-x_1^k))+\frac{1}{\mu}\langle x_1^k-z^k, x-x_1^k\rangle+\frac{1}{2t}\|x-x_1^k\|^2\right\}$}
\State{$x_2^{k+1}=\argmin\limits_{x} \left\{h_2(g_2(x^k_2)+\nabla g_2(x^k_2)^\top (x-x_2^k))+\frac{1}{\mu}\langle x_2^k-z^k, x-x_2^k\rangle+\frac{1}{2t}\|x-x_2^k\|^2\right\}$}
\State{$z^{k+1} = z^k - \gamma\mu^{-1}(x_1^{k+1}-x_2^{k+1})$}
\EndFor
\textbf{end for}
\State{\textbf{return} $x_1^{\tau+1}$ with $\tau$ uniformly sampled from $\{1, \cdots, K\}$}
\end{algorithmic}
\end{algorithm}

Before analyzing the convergence of Algorithm \ref{alg:plm}, we need the following lemma and proposition.
\begin{lemma}[Lemma 3.2 in \cite{Drusvyatskiy2019}]\label{lemma:linear}
    Let  $f_i(z,y):=h_i\left(g_i(y)+\nabla g_i(y)(z-y)\right)$ for $i=1$ and $2$. For any $z,y\in \mathbb{R}^d$, 
    \begin{equation*}
        -\frac{\rho}{2}\|z-y\|^2\leq f_i(z)-f_i(z,y)\leq \frac{\rho}{2}\|z-y\|^2.
    \end{equation*}
\end{lemma}
\begin{proposition}\label{proposition}
Let $\mu^{-1}>\rho$ and $t^{-1}\geq \mu^{-1}+\rho$,
we have
    \begin{equation*}
    \|x_i^{k+1}-x_i^*(z^k)\|^2 \leq (1-tc)\|x_i^{k}-x_i^*(z^k)\|^2, \quad i=1,2,
    \end{equation*}
    where $c:=\mu^{-1}-\rho$.
\end{proposition}
\begin{proof}
The proof is modified from the proof of \cite[Theorem 2.3.4]{nesterov2018}. For simplicity, we omit the subscript $i$ in $h_i$, $g_i$, $x_i^{k+1}$ and $x_i^*(z^k)$, and define the following notations without reflecting their dependency on $i$, $z^k$ and $\mu$:
    \begin{equation*}
    \begin{split}
        F(x) &:= h(g(x))+\frac{1}{2\mu}\|x-z^k\|^2,\\
        F(\Bar{x};x) &:= h(g(\Bar{x})+\nabla g(\Bar{x})(x-\Bar{x})) + \frac{1}{\mu}\langle \Bar{x}-z^k,x-\Bar{x}\rangle,\\
        F_t(\Bar{x};x)& := F(\Bar{x};x)+\frac{1}{2t}\|x-\Bar{x}\|^2,\\
        F^*(\Bar{x};t)&:= \min_x F_t(\Bar{x};x),\\
        x_F &:= \argmin_x F_t(\Bar{x};x),\\
        g_F &:= \frac{1}{t}(\Bar{x}-x_F).
    \end{split}
\end{equation*}
Then we have
\begin{equation}\label{eqn:middle}
    \begin{split}
       F(\Bar{x};x) &= F_t(\Bar{x};x)-\frac{1}{2t}\|x-\Bar{x}\|^2\\
       & \geq F_t(\Bar{x};x_F)+\frac{1}{2t}\|x-x_F\|^2-\frac{1}{2t}\|x-\Bar{x}\|^2\\
       & = F^*(\Bar{x};t)+\frac{1}{2t}\langle \Bar{x}-x_F,2x-x_F-\Bar{x} \rangle\\
       & = F^*(\Bar{x};t) + \langle g_F, x-\Bar{x} \rangle + \frac{t}{2}\| g_F\|^2,
    \end{split}
\end{equation}
where the first inequality is because $F_t(\Bar{x};x)$ is $\frac{1}{t}$-strongly convex. Moreover,
\begin{equation}
    \begin{split}
      F^*(\Bar{x};t) = & F(\Bar{x};x_F)+\frac{1}{2t}\|x_F-\Bar{x}\|^2\\
      =& h(g(\Bar{x})+\nabla g(\Bar{x})(x_F-\Bar{x})) + \frac{1}{\mu}\langle \Bar{x}-z^k,x_F-\Bar{x}\rangle+\frac{1}{2t}\|x_F-\Bar{x}\|^2\\
      \geq & h(g(x_F))-\frac{\rho}{2}\|x_F-\Bar{x}\|^2+\frac{1}{2t}\|x_F-\Bar{x}\|^2\\
      & + \frac{1}{2\mu}\|x_F-z^k\|^2-\frac{1}{2\mu}\|\Bar{x}-z^k\|^2-\frac{1}{2\mu}\|x_F-\Bar{x}\|^2\\\label{eqn:tail}
    \geq & F(x_F)-\frac{1}{2\mu}\|\Bar{x}-z^k\|^2,
    \end{split}
\end{equation}
where the last inequality comes from the Lemma \ref{lemma:linear} and the fact that $\frac{1}{\mu}\langle \Bar{x}-z^k,x_F-\Bar{x}\rangle = \frac{1}{2\mu}\|x_F-z^k\|^2-\frac{1}{2\mu}\|\Bar{x}-z^k\|^2-\frac{1}{2\mu}\|x_F-\Bar{x}\|^2$, and the second inequality is because $t^{-1}\geq \mu^{-1} + \rho$. 
We also have
\begin{equation}
\begin{split}
    F(x) =& h(g(x))+\frac{1}{2\mu}\|x-z^k\|^2\\
        \geq & h(g(\Bar{x})+\nabla g(\Bar{x})(x-\Bar{x}))-\frac{\rho}{2}\|x-\Bar{x}\|^2\\
        & +\frac{1}{2\mu}\|\Bar{x}-z^k\|^2+\frac{1}{\mu}\langle \Bar{x}-z^k,x-\Bar{x} \rangle + \frac{1}{2\mu}\|x-\Bar{x}\|^2\\\label{eqn:head}
        = & F(\Bar{x};x)+\frac{1}{2\mu}\|\Bar{x}-z^k\|^2+\frac{\mu^{-1}-\rho}{2}\|x-\Bar{x}\|^2,
\end{split}
\end{equation}
where the first inequality comes from the Lemma \ref{lemma:linear} and the fact that $\frac{1}{2\mu}\|x-z^k\|^2=\frac{1}{2\mu}\|\Bar{x}-z^k\|^2+\frac{1}{\mu}\langle \Bar{x}-z^k,x-\Bar{x} \rangle + \frac{1}{2\mu}\|x-\Bar{x}\|^2$. 
Combine inequalities (\ref{eqn:middle}), (\ref{eqn:tail}) and (\ref{eqn:head}), we can get
\begin{equation}\label{eqn:F}
    F(x)\geq F(x_F)+\langle g_F,x-\Bar{x} \rangle+\frac{t}{2}\|g_F\|^2 + \frac{\mu^{-1}-\rho}{2}\|x-\Bar{x}\|^2.
\end{equation}
Recall that
\begin{equation*}
    x^*(z^k) = \argmin_x F(x).
\end{equation*}
Let $x = x^*(z^k)$ and $\bar{x}=x^k$ in (\ref{eqn:F}) so $x_F=x^{k+1}$ and $g_F=(x^k-x^{k+1})/t$. Since $F(x_F)-F(x^*(z^k))\geq 0$, it is implied by \eqref{eqn:F} that
\begin{equation}\label{eqn:main}
    \langle g_F,x^k-x^*(z^k) \rangle \geq \frac{t}{2}\|g_F\|^2 + \frac{\mu^{-1}-\rho}{2}\|x^*(z^k)-x^k\|^2.
\end{equation}
The optimality condition of the subproblem $\min_x F_t(\Bar{x};x)$ reads
\begin{equation*}
    g_F\in \nabla g(x^k)^*\partial h(g(x^k)+\nabla g(x^k)(x^{k+1}-x^k))+\frac{1}{\mu}(x^k-z^k).
\end{equation*}
Then, in view of (\ref{eqn:main}), 
we have
\begin{equation}
    \begin{split}
        \|x^{k+1} - x^*(z^k)\|^2 =& \|x_k-x^*(z^k)-t g_F\|^2\\
        = & \|x^{k} - x^*(z^k)\|^2-2t\langle g_F,x^k-x^*(z^k)\rangle + t^2\|g_F\|^2\\
        \leq & (1-t(\mu^{-1}-\rho))\|x^{k} - x^*(z^k)\|^2\\
        \leq & (1-tc)\|x^{k} - x^*(z^k)\|^2
    \end{split}
\end{equation}

\end{proof}

The convergence property of Algorithm \ref{alg:plm} is presented as follows.
\begin{theorem}
    Suppose the assumptions hold and set $\mu^{-1}>\max\{1,\rho\}$, $t^{-1}\geq \mu^{-1}+\rho$, $\gamma=\min\left\{\frac{1}{4L_{\mu}},\sqrt{\frac{t^3c^4\mu^3}{48(1-t^2c^2)}}\right\} $ and $K=O(1/\epsilon^2)$ with $c:=\mu^{-1}-\rho$, then the output of Algorithm \ref{alg:plm} is a stochastic nearly 2$\epsilon$-critical point of (\ref{eqn:prob2}).
\end{theorem}
\begin{proof}
Recall that
\begin{equation*}
    \nabla f_{\mu}(z) = \mu^{-1}(x_1^*(z)-x_2^*(z))
\end{equation*}
and $\nabla f_{\mu}(z)$ is $L_{\mu}$-Lipschitz continuous, we have
\begin{equation*}
    \begin{split}
        &f_{\mu}(z^k) - f_{\mu}(z^{k+1})\\
        \geq & \langle -\nabla f_{\mu}(z^k),z^{k+1}-z^k \rangle -\frac{L_{\mu}}{2}\|z^{k+1}-z^k\|\\
         = & \langle \mu^{-1}(x_1^*(z^k)-x_2^*(z^k)), \gamma\mu^{-1}(x_1^{k+1}-x_2^{k+1}) \rangle - \frac{\gamma^2 L_{\mu}}{2\mu^2}\| x_1^{k+1}-x_2^{k+1} \|^2\\
         = & \frac{\gamma}{\mu^2}\langle x_1^*(z^k)-x_2^*(z^k), x_1^{k+1}-x_2^{k+1}-x_1^*(z^k)+x_2^*(z^k)+x_1^*(z^k)-x_2^*(z^k)\rangle\\
         & - \frac{\gamma^2 L_{\mu}}{2\mu^2}\| x_1^{k+1}-x_2^{k+1} -x_1^*(z^k)+x_2^*(z^k)+x_1^*(z^k)-x_2^*(z^k)\|^2\\
         \geq & \frac{\gamma}{\mu^2}\langle x_1^*(z^k)-x_2^*(z^k), x_1^{k+1}-x_2^{k+1}-x_1^*(z^k)+x_2^*(z^k)\rangle\\
         & +\left( \frac{\gamma}{\mu^2}- \frac{\gamma^2 L_{\mu}}{\mu^2}\right)\|x_1^*(z^k)-x_2^*(z^k)\|^2-\frac{\gamma^2 L_{\mu}}{\mu^2}\| x_1^{k+1}-x_2^{k+1} -x_1^*(z^k)+x_2^*(z^k)\|^2
    \end{split}
\end{equation*}
Applying Young's inequality to the first term on the right-hand side of the last inequality above gives
\begin{equation*}
    \begin{split}
        &f_{\mu}(z^k) - f_{\mu}(z^{k+1})\\
        \geq & -\frac{\gamma}{\mu^2}\left(\frac{1}{2}\|x_1^*(z^k)-x_2^*(z^k)\|^2+\frac{1}{2}\| x_1^{k+1}-x_2^{k+1}-x_1^*(z^k)+x_2^*(z^k) \|^2 \right)\\
        & +\left( \frac{\gamma}{\mu^2}- \frac{\gamma^2 L_{\mu}}{\mu^2}\right)\|x_1^*(z^k)-x_2^*(z^k)\|^2-\frac{\gamma^2 L_{\mu}}{\mu^2}\| x_1^{k+1}-x_2^{k+1} -x_1^*(z^k)+x_2^*(z^k)\|^2\\
        =& \left( \frac{\gamma}{2\mu^2}- \frac{\gamma^2 L_{\mu}}{\mu^2}\right)\|x_1^*(z^k)-x_2^*(z^k)\|^2-\left( \frac{\gamma}{2\mu^2}+ \frac{\gamma^2 L_{\mu}}{\mu^2}\right)\| x_1^{k+1}-x_2^{k+1} -x_1^*(z^k)+x_2^*(z^k)\|^2\\
        \geq & \frac{\gamma}{4\mu^2}\|x_1^*(z^k)-x_2^*(z^k)\|^2 - \frac{3\gamma}{4\mu^2}\| x_1^{k+1}-x_2^{k+1} -x_1^*(z^k)+x_2^*(z^k)\|^2
    \end{split}
\end{equation*}
where the last inequality comes from $\gamma \leq \frac{1}{4L_{\mu}}$. Rearrange terms, we can get
\begin{equation}\label{eqn:diff_prox}
    \begin{split}
        \frac{\gamma}{4\mu^2}\|x_1^*(z^k)-x_2^*(z^k)\|^2 \leq& \frac{3\gamma}{4\mu^2}\| x_1^{k+1}-x_2^{k+1} -x_1^*(z^k)+x_2^*(z^k)\|^2 + f_{\mu}(z^k) - f_{\mu}(z^{k+1})\\
        \leq & \frac{3\gamma}{2\mu^2}\left(\|x_1^{k+1} -x_1^*(z^k)\|^2 + \|x_2^{k+1} -x_2^*(z^k)\|^2\right)+ f_{\mu}(z^k) - f_{\mu}(z^{k+1})
    \end{split}
\end{equation}
Let $\theta:=tc$ and $\alpha$ be a positive constant whose value will be determined later. By Proposition \ref{proposition}, we have
\begin{equation*}
    \begin{split}
        \| x_1^{k+1}-x_1^*(z^k)\|^2 \leq &(1-tc)\| x_1^{k}-x_1^*(z^k)\|^2\\
        =&(1-tc)\| x_1^{k}-x_1^*(z^{k-1})+x_1^*(z^{k-1})-x_1^*(z^{k})\|^2\\
  (\text{By Young's inequality}) \quad     
  \leq & (1+\theta)(1-tc)\|x_1^{k}-x_1^*(z^{k-1})\|^2+\left(1+\frac{1}{\theta}\right)(1-tc)\|x_1^*(z^{k-1})-x_1^*(z^{k})\|^2\\
   (\theta=tc) \quad     \leq & (1-\theta^2)\|x_1^{k}-x_1^*(z^{k-1})\|^2+\left(1+\frac{1}{\theta}\right)(1-\theta)\sigma \|z^{k-1}-z^k\|^2\quad \\
   (\text{By Lemma \ref{lemma_prox} }\sigma=1/\mu c)  \quad   = & (1-\theta^2)\|x_1^{k}-x_1^*(z^{k-1})\|^2+\left(1+\frac{1}{\theta}\right)(1-\theta)\sigma\frac{\gamma^2}{\mu^2}\|x_1^{k}-x_2^{k}\|^2\\
        = &(1-\theta^2)\|x_1^{k}-x_1^*(z^{k-1})\|^2\\
        & +\left(1+\frac{1}{\theta}\right)(1-\theta)\sigma\frac{\gamma^2}{\mu^2}\|x_1^{k}-x_1^*(z^{k-1})-x_2^{k}+x_2^*(z^{k-1})+x_1^*(z^{k-1})-x_2^*(z^{k-1})\|^2\\
    (\text{By Young's inequality}) \quad 
    \leq & (1-\theta^2)\|x_1^{k}-x_1^*(z^{k-1})\|^2\\
        & +\frac{1-\theta^2}{\theta}\sigma\frac{\gamma^2}{\mu^2}(1+\alpha)\|x_1^{k}-x_1^*(z^{k-1})-x_2^{k}+x_2^*(z^{k-1})\|^2\\
        & +\frac{1-\theta^2}{\theta}\sigma\frac{\gamma^2}{\mu^2}(1+1/\alpha)\|x_1^*(z^{k-1})-x_2^*(z^{k-1})\|^2\\
   (\text{By Young's inequality}) \quad
   \leq & (1-\theta^2)\|x_1^{k}-x_1^*(z^{k-1})\|^2\\
        & +\frac{1-\theta^2}{\theta}\sigma\frac{\gamma^2}{\mu^2}2(1+\alpha)\left(\|x_1^{k}-x_1^*(z^{k-1})\|^2+\|x_2^{k}+x_2^*(z^{k-1})\|^2\right)\\
        & + \frac{1-\theta^2}{\theta}\sigma\frac{\gamma^2}{\mu^2}(1+1/\alpha)\|x_1^*(z^{k-1})-x_2^*(z^{k-1})\|^2.
    \end{split}
\end{equation*}
Suppose we can choose $\gamma$ such that 
$$
\gamma^2\leq\frac{\theta^3\mu^2}{6(1-\theta^2)\sigma(1+\alpha)}=\frac{\theta^3\mu^3c}{6(1-\theta^2)(1+\alpha)}.
$$
We will have 
\begin{equation*}
    \frac{\theta^2}{3}\geq \frac{1-\theta^2}{\theta}\sigma\frac{\gamma^2}{\mu^2}2(1+\alpha)
\end{equation*}
and
\begin{equation*}
    \frac{1-\theta^2}{\theta}\sigma\frac{\gamma^2}{\mu^2}(1+1/\alpha) \leq \frac{\theta^2(1+1/\alpha)}{6(1+\alpha)}.
\end{equation*}
Hence,
\begin{equation}
    \begin{split}
         \| x_1^{k+1}-x_1^*(z^k)\|^2 \leq &(1-\theta^2)\|x_1^{k}-x_1^*(z^{k-1})\|^2+\frac{\theta^2}{3}\|x_1^{k}-x_1^*(z^{k-1})\|^2+\frac{\theta^2}{3}\|x_2^{k}-x_2^*(z^{k-1})\|^2\\
         &+\frac{\theta^2(1+1/\alpha)}{6(1+\alpha)}\|x_1^*(z^{k-1})-x_2^*(z^{k-1})\|^2.
    \end{split}
\end{equation}
Similarly, we can get
\begin{equation}
    \begin{split}
         \| x_2^{k+1}-x_2^*(z^k)\|^2 \leq &(1-\theta^2)\|x_2^{k}-x_2^*(z^{k-1})\|^2+\frac{\theta^2}{3}\|x_1^{k}-x_1^*(z^{k-1})\|^2+\frac{\theta^2}{3}\|x_2^{k}-x_2^*(z^{k-1})\|^2\\
         &+\frac{\theta^2(1+1/\alpha)}{6(1+\alpha)}\|x_1^*(z^{k-1})-x_2^*(z^{k-1})\|^2.
    \end{split}
\end{equation}
Add them to get
\begin{equation}\label{eqn:point_prox}
    \begin{split}
        &\| x_1^{k+1}-x_1^*(z^k)\|^2+\| x_2^{k+1}-x_2^*(z^k)\|^2\\
       \leq & \left(1-\frac{\theta^2}{3}\right)\left( \| x_1^{k}-x_1^*(z^{k-1})\|^2+\| x_2^{k}-x_2^*(z^{k-1})\|^2 \right) +\frac{\theta^2(1+1/\alpha)}{3(1+\alpha)}\|x_1^*(z^{k-1})-x_2^*(z^{k-1})\|^2.
    \end{split}
\end{equation}
Let $\alpha = 7$, $\delta_{k}:=\| x_1^{k+1}-x_1^*(z^k)\|^2+\| x_2^{k+1}-x_2^*(z^k)\|^2$, $\Delta_k:=\|x_1^*(z^{k})-x_2^*(z^{k})\|^2$ and sum (\ref{eqn:diff_prox}) and $\frac{147\gamma}{32\mu^2\theta^2}\times$(\ref{eqn:point_prox}), we can get
\begin{equation*}
    \begin{split}
        \frac{\gamma}{\mu^2}\left(\frac{1}{4}\Delta_k-\frac{7}{32}\Delta_{k-1}+\frac{49}{32}\delta_{k-1}-\frac{3}{2}\delta_k\right)\leq \frac{147\gamma}{32\mu^2\theta^2}(\delta_{k-1}-\delta_k)+f_{\mu}(z^k)-f_{\mu}(z^{k+1}).
    \end{split}
\end{equation*}
Taking the sum over $k=1,\cdots,K$, dividing both sides by $K$ and then multiplying both sides by $\frac{32\mu^2}{\gamma}$, we can get
\begin{equation}
    \frac{1}{K}\sum^K_{k=1}(\Delta_k + \delta_k) \leq \frac{1}{K}\left(\frac{147}{\theta^2}\delta_{0}+7\Delta_0+49\Delta_K+\frac{32\mu^2}{\gamma}(f_{\mu}(z^1)-f_{\mu}(z^{K+1}))\right).
\end{equation}
Then we would like to bound $\Delta_K$. We have
\begin{equation*}
\begin{split}
    \Delta_K &= \|x_1^*(z^K)-x_2^*(z^K)\|^2\\
            & = \|x_1^*(z^K)-z^K+z^K-x_2^*(z^K)\|^2\\
            &\leq 2\|x_1^*(z^K)-z^K\|^2+2\|x_2^*(z^K)-z^K\|^2
\end{split}
\end{equation*}
We assume that there exists a constant $C_1$ such that $h_i(g_i(z^K))-\inf_x h_i(g_i(x))\leq C_1$ for any $x$ and $i=1,2$. Then by the definition of $x_i^*(z^K)$, we have for $i=1,2$,
\begin{equation*}
    h_i(g_i(x_i^*(z^K)))+\frac{1}{2\mu}\|x_i^*(z^K)-z^K\|^2\leq h_i(g_i(z^K)).
\end{equation*}
By rearrangement of the above inequality, we have
\begin{equation*}
\begin{split}
    \|x_i^*(z^K)-z^K\|^2&\leq 2\mu \left( h_i(g_i(z^K)) - h_i(g_i(x_i^*(z^K))) \right)\\
    & \leq 2\mu \left( h_i(g_i(z^K)) - \inf_xh_i(g_i(x)) \right) \leq 2\mu C_1.
\end{split}
\end{equation*}
Hence, we have $\Delta_K \leq 8\mu C_1$.

We also assume there is a constant $C$ such that $f_{\mu}(z^1)-f_{\mu}(z^{K+1})\leq C$, then by the definition of $x_{\tau}$, we have
\begin{equation}
    \E\left[ \Delta_{\tau}+\delta_{\tau} \right]\leq \frac{1}{K}\left( \frac{147\delta_0}{\theta^2}+7\Delta_0+392\mu C_1+\frac{32\mu^2 C}{\gamma} \right).
\end{equation}
By taking $K=\max\left\{ \frac{704\delta_0}{\theta^2} , 28\Delta_0, 1568\mu C_1, \frac{96\mu^2 C}{\gamma}\right\}\frac{1}{\mu^2 \epsilon^2}$, we have
\begin{equation*}
    \begin{split}
        \E[\|x_1^{\tau+1}-x_1^*(z^\tau)\|]\leq \mu\epsilon,\quad \E[\|x_2^{\tau+1}-x_2^*(z^\tau)\|]\leq \mu\epsilon,\quad  \E[\|x_1^*(z^\tau)-x_2^*(z^\tau)\|]\leq \mu\epsilon.
    \end{split}
\end{equation*}
By the definition of $x_1^*(z^{\tau})$ and $x_2^*(z^{\tau})$, we have
\begin{equation*}
    \begin{split}
        \frac{1}{\mu}(z^\tau - x_1^*(z^{\tau}))\in \partial f_1(x_1^*(z^{\tau})),\quad \frac{1}{\mu}(z^\tau - x_2^*(z^{\tau}))\in \partial f_2(x_2^*(z^{\tau})).
    \end{split}
\end{equation*}
Let $\xi_1:= \frac{1}{\mu}(z^\tau - x_1^*(z^{\tau}))$ and $\xi_2:= \frac{1}{\mu}(z^\tau - x_2^*(z^{\tau}))$, we have
\begin{equation*}
    \begin{split}
        \E[\|\xi_1-\xi_2\|]=\frac{1}{\mu}\E[\|x_2^*(z^{\tau})-x_1^*(z^{\tau})\|]\leq \epsilon,\quad \E[\|x_1^{\tau+1}-x_1^*(z^\tau)\|]\leq \mu\epsilon,\quad \E[\|x_2^{\tau+1}-x_2^*(z^\tau)\|]\leq \mu\epsilon.
    \end{split}
\end{equation*}
By setting $\mu \leq \min\{1, 1/\rho\}$, we have shown that $$\max\left\{\E[\|\xi_1-\xi_2\|], \E[\|x_1^{\tau+1}-x_1^*(z^\tau)\|], \E[\|x_2^{\tau+1}-x_2^*(z^\tau)\|]\right\}\leq \epsilon,$$ which means that $x_1^{\tau+1}$ is a stochastic nearly 2$\epsilon$-critical point of (\ref{eqn:prob2}).

\end{proof}




\bibliographystyle{plain}
\bibliography{composition}

\begin{thebibliography}{10}

\bibitem{abbaszadehpeivasti2024rate}
Hadi Abbaszadehpeivasti, Etienne de~Klerk, and Moslem Zamani.
\newblock On the rate of convergence of the difference-of-convex algorithm (dca).
\newblock {\em Journal of Optimization Theory and Applications}, 202(1):475--496, 2024.

\bibitem{ahn2017difference}
Miju Ahn, Jong-Shi Pang, and Jack Xin.
\newblock Difference-of-convex learning: directional stationarity, optimality, and sparsity.
\newblock {\em SIAM Journal on Optimization}, 27(3):1637--1665, 2017.

\bibitem{alcantara2024}
Jan~Harold Alcantara, Ching pei Lee, and Akiko Takeda.
\newblock A four-operator splitting algorithm for nonconvex and nonsmooth optimization, 2024.

\bibitem{alvarado2014new}
Alberth Alvarado, Gesualdo Scutari, and Jong-Shi Pang.
\newblock A new decomposition method for multiuser dc-programming and its applications.
\newblock {\em IEEE Transactions on Signal Processing}, 62(11):2984--2998, 2014.

\bibitem{benaim2005}
Michel Bena\"{\i}m, Josef Hofbauer, and Sylvain Sorin.
\newblock Stochastic approximations and differential inclusions.
\newblock {\em SIAM Journal on Control and Optimization}, 44(1):328--348, 2005.

\bibitem{bolte2020}
Jérôme Bolte and Edouard Pauwels.
\newblock Conservative set valued fields, automatic differentiation, stochastic gradient method and deep learning, 2020.

\bibitem{Bolte2018}
Jérôme Bolte, Shoham Sabach, and Marc Teboulle.
\newblock Nonconvex lagrangian-based optimization: Monitoring schemes and global convergence.
\newblock {\em Mathematics of operations research}, 43(4):1210--1232, 2018.

\bibitem{cao2022unifying}
Shanshan Cao, Xiaoming Huo, and Jong-Shi Pang.
\newblock A unifying framework of high-dimensional sparse estimation with difference-of-convex (dc) regularizations.
\newblock {\em Statistical Science}, 37(3):411--424, 2022.

\bibitem{cui2018composite}
Ying Cui, Jong-Shi Pang, and Bodhisattva Sen.
\newblock Composite difference-max programs for modern statistical estimation problems.
\newblock {\em SIAM Journal on Optimization}, 28(4):3344--3374, 2018.

\bibitem{davis2019}
Damek Davis and Dmitriy Drusvyatskiy.
\newblock Stochastic model-based minimization of weakly convex functions.
\newblock {\em SIAM Journal on Optimization}, 29(1):207--239, 2019.

\bibitem{Davis2020}
Damek Davis, Dmitriy Drusvyatskiy, Sham Kakade, and Jason~D. Lee.
\newblock Stochastic subgradient method converges on tame functions.
\newblock {\em Foundations of Computational Mathematics}, 20(1):119--154, Feb 2020.

\bibitem{Drusvyatskiy2021}
D.~Drusvyatskiy, A.~D. Ioffe, and A.~S. Lewis.
\newblock Nonsmooth optimization using taylor-like models: error bounds, convergence, and termination criteria.
\newblock {\em Mathematical Programming}, 185(1):357--383, Jan 2021.

\bibitem{Drusvyatskiy2019}
D.~Drusvyatskiy and C.~Paquette.
\newblock Efficiency of minimizing compositions of convex functions and smooth maps.
\newblock {\em Mathematical Programming}, 178(1):503--558, Nov 2019.

\bibitem{Drusvyatskiy2018}
Dmitriy Drusvyatskiy and Adrian~S. Lewis.
\newblock Error bounds, quadratic growth, and linear convergence of proximal methods.
\newblock {\em Math. Oper. Res.}, 43(3):919–948, August 2018.

\bibitem{duchi2018}
John~C. Duchi and Feng Ruan.
\newblock Stochastic methods for composite and weakly convex optimization problems.
\newblock {\em SIAM Journal on Optimization}, 28(4):3229--3259, 2018.

\bibitem{Goldstein1977}
A.~A. Goldstein.
\newblock Optimization of lipschitz continuous functions.
\newblock {\em Mathematical Programming}, 13(1):14--22, Dec 1977.

\bibitem{hu2024}
Quanqi Hu, Qi~Qi, Zhaosong Lu, and Tianbao Yang.
\newblock Single-loop stochastic algorithms for difference of max-structured weakly convex functions.
\newblock In {\em Advances in Neural Information Processing Systems 37: Annual Conference on Neural Information Processing Systems 2024, NeurIPS 2024}, 2024.

\bibitem{hu23}
Quanqi Hu, Dixian Zhu, and Tianbao Yang.
\newblock Non-smooth weakly-convex finite-sum coupled compositional optimization.
\newblock In A.~Oh, T.~Naumann, A.~Globerson, K.~Saenko, M.~Hardt, and S.~Levine, editors, {\em Advances in Neural Information Processing Systems}, volume~36, pages 5348--5403. Curran Associates, Inc., 2023.

\bibitem{jordan2023deterministic}
Michael Jordan, Guy Kornowski, Tianyi Lin, Ohad Shamir, and Manolis Zampetakis.
\newblock Deterministic nonsmooth nonconvex optimization.
\newblock In {\em The Thirty Sixth Annual Conference on Learning Theory}, pages 4570--4597. PMLR, 2023.

\bibitem{kiw07}
Krzysztof~C. Kiwiel.
\newblock Convergence of the gradient sampling algorithm for nonsmooth nonconvex optimization.
\newblock {\em SIAM Journal on Optimization}, 18(2):379--388, 2007.

\bibitem{kor22}
Guy Kornowski and Ohad Shamir.
\newblock Oracle complexity in nonsmooth nonconvex optimization.
\newblock {\em Journal of Machine Learning Research}, 23(314):1--44, 2022.

\bibitem{thi2024}
Hoai~An Le~Thi, Van~Ngai Huynh, and Tao~Pham Dinh.
\newblock Minimizing compositions of differences-of-convex functions with smooth mappings.
\newblock {\em Math. Oper. Res.}, 49(2):1140–1168, July 2023.

\bibitem{le2018convergence}
Hoai~An Le~Thi, Van~Ngai Huynh, and Tao Pham~Dinh.
\newblock Convergence analysis of difference-of-convex algorithm with subanalytic data.
\newblock {\em Journal of Optimization Theory and Applications}, 179(1):103--126, 2018.

\bibitem{Lewis2016}
A.~S. Lewis and S.~J. Wright.
\newblock A proximal method for composite minimization.
\newblock {\em Mathematical Programming}, 158(1):501--546, Jul 2016.

\bibitem{liu2020two}
Junyi Liu, Ying Cui, Jong-Shi Pang, and Suvrajeet Sen.
\newblock Two-stage stochastic programming with linearly bi-parameterized quadratic recourse.
\newblock {\em SIAM Journal on Optimization}, 30(3):2530--2558, 2020.

\bibitem{liu2023risk}
Junyi Liu and Jong-Shi Pang.
\newblock Risk-based robust statistical learning by stochastic difference-of-convex value-function optimization.
\newblock {\em Operations Research}, 71(2):397--414, 2023.

\bibitem{pong2019}
Tianxiang Liu, Ting~Kei Pong, and Akiko Takeda.
\newblock A successive difference-of-convex approximation method for a class of nonconvex nonsmooth optimization problems.
\newblock {\em Math. Program.}, 176(1–2):339–367, July 2019.

\bibitem{majewski2018}
Szymon Majewski, Błażej Miasojedow, and Eric Moulines.
\newblock Analysis of nonsmooth stochastic approximation: the differential inclusion approach, 2018.

\bibitem{moudaficomplete}
Abdellatif Moudafi.
\newblock A complete smooth regularization of dc optimization problems.

\bibitem{nesterov2018}
Yurii Nesterov.
\newblock {\em Lectures on Convex Optimization}.
\newblock Springer Publishing Company, Incorporated, 2nd edition, 2018.

\bibitem{nouiehed2019pervasiveness}
Maher Nouiehed, Jong-Shi Pang, and Meisam Razaviyayn.
\newblock On the pervasiveness of difference-convexity in optimization and statistics.
\newblock {\em Mathematical Programming}, 174(1):195--222, 2019.

\bibitem{pang2017computing}
Jong-Shi Pang, Meisam Razaviyayn, and Alberth Alvarado.
\newblock Computing b-stationary points of nonsmooth dc programs.
\newblock {\em Mathematics of Operations Research}, 42(1):95--118, 2017.

\bibitem{pang2018decomposition}
Jong-Shi Pang and Min Tao.
\newblock Decomposition methods for computing directional stationary solutions of a class of nonsmooth nonconvex optimization problems.
\newblock {\em SIAM Journal on Optimization}, 28(2):1640--1669, 2018.

\bibitem{Rock98}
{R. Tyrrell} Rockafellar and Roger J.-B. Wets.
\newblock {\em Variational Analysis}.
\newblock Springer Verlag, Heidelberg, Berlin, New York, 1998.

\bibitem{Sun2022}
Kaizhao Sun and Xu~Andy Sun.
\newblock Algorithms for difference-of-convex programs based on difference-of-moreau-envelopes smoothing.
\newblock {\em INFORMS J. Optim.}, 5:321--339, 2022.

\bibitem{tao1997convex}
Pham~Dinh Tao and LT~Hoai An.
\newblock Convex analysis approach to dc programming: theory, algorithms and applications.
\newblock {\em Acta mathematica vietnamica}, 22(1):289--355, 1997.

\bibitem{tian2024no}
Lai Tian and Anthony Man-Cho So.
\newblock No dimension-free deterministic algorithm computes approximate stationarities of lipschitzians.
\newblock {\em Mathematical Programming}, pages 1--24, 2024.

\bibitem{wang2024alexr}
Bokun Wang and Tianbao Yang.
\newblock Alexr: An optimal single-loop algorithm for convex finite-sum coupled compositional stochastic optimization, 2024.

\bibitem{wang2024}
Xiao Wang and Xiaojun Chen.
\newblock Complexity of finite-sum optimization with nonsmooth composite functions and non-lipschitz regularization.
\newblock {\em SIAM Journal on Optimization}, 34(3):2472--2502, 2024.

\bibitem{Xu2018}
Yi~Xu, Qi~Qi, Qihang Lin, Rong Jin, and Tianbao Yang.
\newblock Stochastic optimization for dc functions and non-smooth non-convex regularizers with non-asymptotic convergence.
\newblock In {\em International Conference on Machine Learning}, 2018.

\bibitem{Yang2024}
Lei Yang.
\newblock Proximal gradient method with extrapolation and line search for a class of non-convex and non-smooth problems.
\newblock {\em Journal of Optimization Theory and Applications}, 200(1):68--103, Jan 2024.

\bibitem{yao22}
Yao Yao, Qihang Lin, and Tianbao Yang.
\newblock Large-scale optimization of partial {AUC} in a range of false positive rates.
\newblock In Sanmi Koyejo, S.~Mohamed, A.~Agarwal, Danielle Belgrave, K.~Cho, and A.~Oh, editors, {\em Advances in Neural Information Processing Systems 35: Annual Conference on Neural Information Processing Systems 2022, NeurIPS 2022, New Orleans, LA, USA, November 28 - December 9, 2022}, 2022.

\bibitem{Zhang2020}
Jiawei Zhang and Zhi-Quan Luo.
\newblock A proximal alternating direction method of multiplier for linearly constrained nonconvex minimization.
\newblock {\em SIAM journal on optimization}, 30(3):2272--2302, 2020.

\bibitem{zhang20}
Jingzhao Zhang, Hongzhou Lin, Stefanie Jegelka, Suvrit Sra, and Ali Jadbabaie.
\newblock Complexity of finding stationary points of nonsmooth nonconvex functions.
\newblock In {\em Proceedings of the 37th International Conference on Machine Learning}, ICML'20. JMLR.org, 2020.

\bibitem{zhang2020optimal}
Zhe Zhang and Guanghui Lan.
\newblock Optimal algorithms for convex nested stochastic composite optimization.
\newblock {\em arXiv preprint arXiv:2011.10076}, 2020.

\end{thebibliography}

\end{document}